\documentclass[preprint]{article}
\usepackage{mathrsfs}
\usepackage{amssymb} % \mathbb
\usepackage{amsmath}
\usepackage{hyperref}
\usepackage[linesnumbered, ruled]{algorithm2e}
\usepackage{multirow}
\usepackage{colortbl}

\newtheorem{theorem}{Theorem}
\newtheorem{proposition}[theorem]{Proposition}
\newtheorem{lemma}[theorem]{Lemma}

\newtheorem{corollary}[theorem]{Corollary}
\newenvironment{proof}{{\bf Proof.}}{\hspace{0.35cm} $\Box$}

\begin{document}
  \author{Xue Jiang ~~~Shugong Zhang\footnote{Institute of Mathematics, Key Lab. of Symbolic Computation and
Knowledge Engineering (Ministry of Education), Jilin University,
Changchun, 130012, PR China}}
  \title{The Breadth-one $D$-invariant Polynomial Subspace}
  \date{}
  \maketitle

\noindent\textbf{Abstract.}~We demonstrate the equivalence of two classes of $D$-invariant polynomial subspaces introduced in [8] and [9], i.e., these two classes of subspaces are  different representations of the breadth-one $D$-invariant subspace. Moreover, we solve the discrete approximation problem in ideal interpolation for the breadth-one $D$-invariant subspace.  Namely, we find the points, such that the limiting space of the evaluation functionals at these points is  the functional space induced by the given $D$-invariant subspace, as the evaluation points all coalesce at one point.\\

\noindent \textbf{Keywords.}~$D$-invariant polynomial  subspace; Breadth-one; Ideal interpolation;  Discrete approximation problem.

\noindent\textbf{MSC (2010).}~41A63, 41A05, 41A10, 41A35.

\section{Introduction}\label{sec:Int}
A polynomial subspace  is said to be $D$-invariant if it is closed under differentiation. The breadth of a $D$-invariant polynomial subspace is defined as the number of all  linear polynomials  in a basis of this space. This number may not be unique when different bases are considered, here we choose the maximum  as the breadth of the space. The breadth-one $D$-invariant subspace is mentioned by Dayton and Zeng in \cite{Zeng,Dayton}, and is further discussed by Li and Zhi in \cite{LiNan} where a recursion formula of each polynomial in the basis  is given. We provide another class of $D$-invariant subspaces in \cite{LiZhe}.

% Studying the structure of  $D$-invariant subspace is a very important topic in symbolic computation, especially in
%Ideal interpolation and polynomial system solving are fundamental problems in algebra geometry in which $D$-invariant subspace is  of much concern.

Ideal interpolation is originated in the paper of  G. Birkhoff  \cite{Birkhoff} where the interpolation problem is defined by a linear projector. C. de Boor and B. Shekhtman survey some results and raise some problems in their review articles \cite{de boor2005} and \cite{Boris09new}, respectively.
In ideal interpolation, the interpolation conditions at an interpolation site $\mathbf{z}$ can be described by a space of linear functionals, i.e.,
$\mathrm{span}\{\delta_{\mathbf{z}}\circ p(D), p\in P_{\mathbf{z}}\}$, where $P_{\mathbf{z}}$ is a $D$-invariant polynomial subspace, $\delta_{\mathbf{z}}$ is the evaluation functional at $\mathbf{z}$ and $p(D)$ is the  differential operator induced by $p$. \emph{Lagrange interpolation} is a standard example where all $P_{\mathbf{z}}=\mathrm{span}\{1\}$.  Note that in one variable every ideal interpolation  (over complex field)  is the pointwise  limit of Lagrange interpolation, and this is still true for some  multivariate examples \cite{de boor limit,Boristwovariablelimit,de Boor90}. However,  B. Shekhtman provides counterexamples when there are more than two variables  \cite{Boris06}. In \cite{Boris consappro}, B. Shekhtman prescribes finitely many computations to determine whether a projector is a limit of Lagrange projectors. But as pointed out by the author, ``finitely many" is still far too many steps for a computer to handle even in a very simple case. This is the main motivation for us to consider the problem:

Given an ideal interpolant with its interpolation conditions $\delta_{\mathbf{z}}P_{\mathbf{z}}(D)$, where $P_{\mathbf{z}}$ is a $D$-invariant $n+1$-dimensional polynomial subspace, find $n+1$ points $\mathbf{z}_{0}(h),\mathbf{z}_{1}(h),\dots,\mathbf{z}_{n}(h)$ such that
\begin{align}  \label{formula:appro}
  \lim_{h\rightarrow 0} \mathrm{span}\{\delta_{\mathbf{z}_{0}(h)},\delta_{\mathbf{z}_{1}(h)},\dots,\delta_{\mathbf{z}_{n}(h)}\}= \{\delta_{\mathbf{z}}\circ p(D) :p \in P_{\mathbf{z}} \}.
\end{align}
We call this problem the  \emph{discrete approximation problem} for  $\delta_{\mathbf{z}}P_{\mathbf{z}}(D)$ and $\mathbf{z}_{0}(h)$,\\$\mathbf{z}_{1}(h),\dots,\mathbf{z}_{n}(h)$ the \emph{discrete points} for $\delta_{\mathbf{z}}P_{\mathbf{z}}(D)$.

Actually, this question is first raised by C. de Boor and A. Ron in \cite{de Boor90} where $P_{\mathbf{z}}$ is a $D$-invariant subspace spanned by homogeneous polynomials. We have solved this problem for the case that $P_{\mathbf{z}}$ is a $D$-invariant subspace with maximal total degree two in another paper. In this paper we will solve the problem for a particular case when $P_{\mathbf{z}}$ is a breadth-one $D$-invariant subspace.

The paper is organised as follows. We first  prove that the two classes of $D$-invariant subspaces introduced in  \cite{LiNan} and \cite{LiZhe} are equivalent in the sense of coordinate transformation in Section \ref{sec:3}. Then in Section \ref{sec:4}, we solve the discrete approximation problem for the  breadth-one $D$-invariant subspace in two ways. Namely, we present two sets of discrete points for this special subspace. The next section is devoted to introducing some notation and  known results.
\section{Preliminaries}\label{sec:Pre}

Throughout the paper, $\mathbb{F}$ denotes a field with characteristic zero. $\mathbb{F}[\mathbf{x}]:=\mathbb{F}[x_{1},\dots,x_{d}]$ denotes the polynomial ring in $d$ variables over  $\mathbb{F}$.  For  $\pmb\alpha=(\alpha_{1},\dots,\alpha_{d})\in \mathbb{N}^{d}_{+}$, $\pmb\alpha!:=\alpha_{1}!\cdots\alpha_{d}!$. For  $p=\sum_{\pmb\alpha}a_{\pmb\alpha}\mathbf{x}^{\pmb\alpha}\in \mathbb{F}[\mathbf{x}]$, $p(D)$ is defined as
$$p(D):=\sum_{\pmb\alpha}a_{\pmb\alpha}\frac{\partial^{|\pmb\alpha|}}{\partial x_{1}^{\alpha_{1}}\cdots \partial x_{d}^{\alpha_{d}}}. $$
We define a operator $\Psi_{j}$ on $\mathbb{F}[\mathbf{x}]$ that acts as ``integral":
$$\Psi_{j}(\mathbf{x}^\alpha):=\frac{1}{\alpha_{j}+1}x_1^{\alpha_{1}}\cdots
x_{j}^{\alpha_{j}+1}\cdots x_{d}^{\alpha_{d}},~~j=1,\dots,d. $$
%A polynomial subspace $P\subset \mathbb{F}[\mathbf{x}]$ is said to be $D$-invariant if it is closed under differentiation, i.e., $\forall p \in P, \frac{\partial p}{\partial x_{j}}\in P$ holds for all $j=1,\dots,d. $
%For an arbitrary \mbox{\boldmath$\alpha$}$=(\alpha_{1},\alpha_{2},\dots,\alpha_{d})\in \mathbb{N}^{d}_{+}$, $|\alpha|=\alpha_{1}+\alpha_{2}+\dots+\alpha_{d}$, we denote by
%$$D^{\alpha}=D^{(\alpha_{1},\alpha_{2},\dots,\alpha_{d})}:=\frac{1}{\alpha_{1}!\alpha_{2}!\dots\alpha_{d}!}\frac{\partial^{\alpha_{1}+\alpha_{2}+\dots+\alpha_{d}}}{\partial x_{1}^{\alpha_{1}}x_{2}^{\alpha_{2}}\dots \partial x_{d}^{\alpha_{d}}}$$
%the differentiation monomial with order $|\alpha|$. Let $p=\sum_{\alpha}a_{\alpha}\mathbf{x}^{\alpha}\in \mathbb{F}[x_{1},\dots,x_{d}]$, then the differential operator $p(D)$ can be written as
%$$p(D):=\sum_{\alpha}a_{\alpha}\alpha_{1}!\dots\alpha_{d}!D^{\alpha}.$$
%The differential operator $\Psi_{j}$ and the anti-differentiation operator $\Phi_{j}$ are defined as
%\begin{align*}
%\Psi_{j}(D^{\alpha}):&=D^{(\alpha_{1},\dots,\alpha_{j}+1,\dots,\alpha_{d})}, \\
%\Phi_{j}(D^{\alpha}):&=\left\{
%                         \begin{array}{ll}
%                           D^{(\alpha_{1},\dots,\alpha_{j}-1,\dots,\alpha_{d})}, & \hbox{$\alpha_{j}>0$;} \\
%                           0, & \hbox{otherwise.}
%                         \end{array}
%                       \right.
%\end{align*}
%$P(D):=\{p(D): p\in P\}$ is the associated linear space of differential operators, and we say that $P(D)$ is $D$-invariant if and only if $P$ is $D$-invariant.

Li and Zhi demonstrate the structure of the  breadth-one  $D$-invariant polynomial subspace as follows:
%\begin{theorem}\cite{LiNan}
% Let $\mathrm{span}\{L_{0},L_{1},\dots,L_{n}\}$ be the breadth-one  $D$-invariant  subspace of differential operators, where $L_{0}=1$, $L_{1}=D^{(1,0,\dots,0)}$. We can construct the $k$-th order differential operator incrementally for $k$ from $2$ to $n$ by the following formula:
%\begin{align}
%  L_{k}= V_{k}+a_{k,2}D^{(0,1,0,\dots,0)}+\dots+a_{k,d}D^{(0,\dots,0,1)},
%\label{formu:DiTui}
%\end{align}
%where $V_{k}$ has no free parameters and is obtained from the computed basis $\{L_{1},\dots,L_{k-1}\}$ by the following formula:
%$$V_{k}=\Psi_{1}(M_{1})+\Psi_{2}((M_{2})_{i_{1}=0})+\dots+\Psi_{d}((M_{d})_{i_{1}=i_{2}=\dots=i_{d-1}=0}),  $$
%where
%$$M_{1}=L_{k-1},~M_{j}=a_{2,j}L_{k-2}+\dots+a_{k-1,j}L_{1},~~~~2\leq j\leq d.$$
%Here $i_{1}=\dots=i_{j-1}=0$ means that we only pick up terms which do not contain derivatives in $\frac{\partial}{\partial x_{1}},\dots,\frac{\partial}{\partial x_{j-1}}$, and  $a_{i,j}$ are known parameters appearing in $L_{i}$ for $2\leq i \leq k-1,~2\leq j \leq d$.
%\label{theo:Linan}
%\end{theorem}
\begin{theorem}\label{theorem:1}
\cite{LiNan} Let $\mathcal{L}_{n}:=\mathrm{span}\{L_{0},L_{1},\dots,L_{n}\}$ be the breadth-one  $D$-invariant polynomial  subspace, where $L_{0}=1$, $L_{1}=x_{1}$. We can construct the $k$-th degree polynomial  incrementally for $k$ from $2$ to $n$ by the following formula:
\begin{align}
  L_{k}= V_{k}+a_{k,2}x_2+\dots+a_{k,d}x_d,
\label{formu:DiTui}
\end{align}
where $V_{k}$ has no free parameters and is obtained from the computed basis $\{L_{1},\dots$,\\$L_{k-1}\}$ by the following formula:
$$V_{k}=\Psi_{1}(M_{1})+\Psi_{2}((M_{2})_{i_{1}=0})+\dots+\Psi_{d}((M_{d})_{i_{1}=i_{2}=\dots=i_{d-1}=0}),  $$
where
$$M_{1}=L_{k-1},~M_{j}=a_{2,j}L_{k-2}+\dots+a_{k-1,j}L_{1},~~~~2\leq j\leq d.$$
Here $i_{1}=\dots=i_{j-1}=0$ means that we only pick up terms which do not contain variables in $x_1,\dots,x_{j-1}$, and  $a_{i,j}\in \mathbb{F}$ are known parameters appearing in $L_{i}$ for $2\leq i \leq k-1,~2\leq j \leq d$.
\label{theo:Linan}
\end{theorem}

Note that $1$ is always in a $D$-invariant subspace and the linear polynomial $L_{1}$ has the form $x_1$  with an appropriate linear coordinate transformation. Thus the hypotheses in Theorem $\ref{theo:Linan}$ are reasonable and every breadth-one $D$-invariant subspace can be written  in the above form with specified parameters $a_{i,j}$. For example, for $d=2$,
%\begin{align*}
%  L_1&=D^{(1,0)};~~ L_2=D^{(2,0)}+a_{2,2}D^{(0,1)};\\
%  L_3&=D^{(3,0)}+a_{2,2}D^{(1,1)}+a_{3,2}D^{(0,1)};\\
%  L_4&=D^{(4,0)}+a_{2,2}D^{(2,1)}+a_{3,2}D^{(1,1)}+a_{2,2}^{2}D^{(0,2)}+a_{4,2}D^{(0,1)};\\
%  &\dots
%\end{align*}
\begin{align*}
  L_1&=x_1;~~ L_2=\frac{1}{2!}x^2_1+a_{2,2}x_2;\\
  L_3&=\frac{1}{3!}x_1^3+a_{2,2}x_1x_2+a_{3,2}x_2;\\
  L_4&=\frac{1}{4!}x_1^4+\frac{1}{2!}a_{2,2}x_1^2x_2+a_{3,2}x_1x_2+\frac{1}{2!}a_{2,2}^{2}x_2^2+a_{4,2}x_2;\\
  &\dots
\end{align*}

To introduce another class of $D$-invariant polynomial subspaces discussed in \cite{LiZhe}, we first need some notation. Let $\textbf{\emph{b}}=(b_{1},b_{2},\dots,b_{n})\in \mathbb{N}^{n}_{+}$, $n\geq 2$ satisfying %~$\textbf{\emph{a}}$~ 的各个分量满足条件
$$b_{1}=1,~b_{n}>\dots>b_{2}\geq 2, $$
and let
$$\textbf{\emph{c}}_{1}=(c_{1,1},c_{1,2},\dots,c_{1,n}),\dots,\textbf{\emph{c}}_{d}=(c_{d,1},c_{d,2},\dots,c_{d,n})\in \mathbb{F}^{n},  $$
where $c_{1,1},\dots, c_{d,1}$~are not all zero. Construct  a map $\tau:(\mathbb{N}^{n})^{d}\rightarrow \mathbb{N}$ defined by
$$\tau(\pmb\gamma_{1},\dots,\pmb\gamma_{d})= \sum_{j=1}^{n}b_{j}\sum_{i=1}^{d}\gamma_{i,j}, $$
with $\pmb\gamma_{i}=(\gamma_{i,1},\gamma_{i,2},\dots,\gamma_{i,n})\in \mathbb{N}^{n},~i=1,\dots,d$.
\begin{theorem} \cite{LiZhe}
  Let $\textbf{b}=(1,b_{2},\dots,b_{n})$, $\textbf{c}_{i}=(c_{i,1},c_{i,2},\dots,c_{i,n})$, $i=1,\dots,d$, and the map $\tau$ be as above. Let $q_{n,m}$ be a polynomial defined by
  \begin{align}
  q_{n,m}=\sum_{\tau(\pmb\gamma_{1},\dots,\pmb\gamma_{d})=m}\frac{\textbf{c}_{1}^{\pmb\gamma_{1}}\cdots \textbf{c}_{d}^{\pmb\gamma_{d}}}{\pmb\gamma_{1}!\cdots\pmb\gamma_{d}!}x_{1}^{|\pmb\gamma_{1}|}\cdots x_{d}^{|\pmb\gamma_{d}|},~~ m=0,1,2,\dots,b_{n},
  \label{formu:qInTheo}
  \end{align}
  with $\pmb\gamma_{i}=(\gamma_{i,1},\gamma_{i,2},\dots,\gamma_{i,n})\in \mathbb{N}^{n},~\textbf{c}_{i}^{\pmb\gamma_{i}}=\prod_{j=1}^{n}c_{i,j}^{\gamma_{i,j}}$. Then the linear  space
  $$Q=\mathrm{span}\{q_{n,m}: m=0,1,2,\dots,b_{n}\}$$
  is a $(b_{n}+1)$-dimensional  $D$-invariant polynomial subspace.
  \label{theo:LIZhe}
\end{theorem}

\section{The Structure of the $D$-invariant Subspace: Case of Breadth One}\label{sec:3}
In this section, we will  prove the equivalence of the two classes of $D$-invariant subspaces introduced in the precious section.% in the sense of coordinate transformation.

%\section{The Equivalence of the Two Classes of $D$-invariant Subspace}
\begin{lemma}  \label{lemma:Qbreadthone}
  The breadth of the $D$-invariant subspace $Q$ given by Theorem $\ref{theo:LIZhe}$  is $1$.
\end{lemma}
\begin{proof}
 By the construction of $q_{n,m}$  and $b_{1}=1$, we know that $\mathrm{deg}(q_{n,m})=m$, for all $m=1,2,\dots,b_{n}$. Namely, there is only one linear polynomial in the basis of $Q$, thus the lemma is  proved.
 \end{proof}\\

We denote by $q_{n,m}^{*}$ the polynomial obtained by specifying
\begin{align*}
   \textbf{\emph{b}\hspace{0.3mm}}&=(b_{1},b_{2},\dots,b_{n})=(1,2,\dots,n),\\
   \textbf{\emph{c}}_{1}&=(c_{1,1},c_{1,2},\dots,c_{1,n})=(1,0,\dots,0),\\
   \textbf{\emph{c}}_{s}&=(c_{s,1},c_{s,2},\dots,c_{s,n})=(0,a_{2,s},a_{3,s},\dots,a_{n,s}),~~ s=2,\dots,d,
\end{align*}
in $q_{n,m}$, where $a_{i,s}, i=2,\dots,n, s=2,\dots,d$ are the parameters in $(\ref{formu:DiTui})$. We define $0^0=1$ in this paper. Since $c_{1,2}=\dots=c_{1,n}=0, c_{2,1}=\dots=c_{d,1}=0$, it follows that $\gamma_{1,2},\dots,\gamma_{1,n}$, $\gamma_{2,1},\dots,\gamma_{d,1}$ must be zero, or the corresponding term in $q_{n,m}$ will be zero, thus
  \begin{align*}
  q_{n,m}^{*}&=\sum_{\tau(\pmb\gamma_{1},\dots,\pmb\gamma_{d})=m}\frac{c_{1,1}^{\gamma_{1,1}} c_{2,2}^{\gamma_{2,2}}\cdots c_{2,n}^{\gamma_{2,n}}\cdots c_{d,2}^{\gamma_{d,2}}\cdots c_{d,n}^{\gamma_{d,n}} }{\gamma_{1,1}!\gamma_{2,2}!\cdots \gamma_{2,n}!\cdots \gamma_{d,2}!\cdots \gamma_{d,n}!}x_{1}^{|\pmb\gamma_{1}|}\cdots x_{d}^{|\pmb\gamma_{d}|}\\
  &=\sum_{\tau(\pmb\gamma_{1},\dots,\pmb\gamma_{d})=m}\frac{1^{\gamma_{1,1}} a_{2,2}^{\gamma_{2,2}}\cdots a_{n,2}^{\gamma_{2,n}}\cdots a_{2,d}^{\gamma_{d,2}}\cdots a_{n,d}^{\gamma_{d,n}} }{\gamma_{1,1}!\gamma_{2,2}!\cdots \gamma_{2,n}!\cdots \gamma_{d,2}!\cdots \gamma_{d,n}!}x_{1}^{|\pmb\gamma_{1}|}\cdots x_{d}^{|\pmb\gamma_{d}|},
  %\label{formu:q}
  \end{align*}
  where $\tau(\pmb\gamma_{1},\dots,\pmb\gamma_{d})$ can be written as
  \begin{align}
  \tau(\pmb\gamma_{1},\dots,\pmb\gamma_{d})=\gamma_{1,1}+2(\gamma_{2,2}+\dots+\gamma_{d,2})+\dots+n(\gamma_{2,n}+\dots+\gamma_{d,n}).
  \label{formu:tau}
  \end{align}

 For $d=2$, one can verify that
 \begin{align*}
 q^{*}_{n,0}&=1;~~q^{*}_{n,1}=x_1;~~q^{*}_{n,2}=\frac{1}{2!}x_1^2+a_{2,2}x_2;\\
 q^{*}_{n,3}&=\frac{1}{3!}x_1^3+a_{2,2}x_1x_2+a_{3,2}x_2;\\
 q^{*}_{n,4}&=\frac{1}{4!}x_1^4+\frac{1}{2!}a_{2,2}x_1^2x_2+\frac{1}{2!}a_{2,2}^2x_2^2+a_{3,2}x_1x_2+a_{4,2}x_2;\\
 &\dots
 \end{align*}
 It is easy to see that for $d=2$, $q^{*}_{n,i}=L_i, i=1,\dots,4$. More generally, we have
\begin{proposition} \label{prop:1}
  With the above notation, $\forall n\geq 0,$
  \begin{align}
  q^{*}_{n,m}=L_{m},~~ m=0,1,\dots,n.
  \label{formu:qInProp}
  \end{align}
\end{proposition}
\begin{proof}
We will use induction on $m$. From the previous discussion $q^{*}_{n,0}=L_{0}, q^{*}_{n,1}=L_{1}$. Now assume that  the proposition is true for all $m\leq n-1$.
 By Lemma $3.4$ and Theorem $3.1$ in \cite{LiNan}, we know that
 \begin{align*}
   \frac{\partial L_n}{\partial x_1}&=L_{n-1};\\
   \frac{\partial L_{n}}{\partial x_j}&=a_{2,j}L_{n-2}+\dots+a_{n,j}L_{0},~~2\leq j\leq d.
 \end{align*}
 By the proof of  Proposition $4.1$ in \cite{LiZhe}, we know that for any $n\geq 2$,
 \begin{align*}
   \frac{\partial q^{*}_{n,n}}{\partial x_{1}}&=c_{1,1}q^{*}_{n,n-1}+\sum_{i=2}^{n}c_{1,i}q^{*}_{n,n-i}=q^{*}_{n,n-1};\\
   \frac{\partial q^{*}_{n,n}}{\partial x_j}&=c_{j,1}q^{*}_{n,n-1}+\sum_{i=2}^{n}c_{j,i}q^{*}_{n,n-i}=a_{2,j}q^{*}_{n,n-2}+\dots+a_{n,j}q^*_{n,0},~2\leq j\leq d.
 \end{align*}
Since $q^*_{n,m}=L_{m}$ for $0\leq m \leq n-1$ by our inductive assumption, it follows that
$$\frac{\partial q^{*}_{n,n}}{\partial x_{j}}=\frac{\partial L_n}{\partial x_j},~~~  1\leq j\leq d.$$
Note that $q_{n,n}^*$ and $L_{n}$ do not contain constant term, thus $q_{n,n}^*=L_n$. This completes the proof.
\end{proof}

\begin{theorem} \label{theo:3}
 The subspace $Q=\mathrm{span}\{q_{n,m}: m=0,1,2,\dots,b_{n}\}$ given in Theorem $\ref{theo:LIZhe}$ is equivalent to the breadth-one $D$-invariant subspace $$\mathcal{L}_{b_n}=\mathrm{span}\{L_{0},L_{1},\dots,L_{b_{n}}\}$$ in the sense of  coordinate transformation.
\end{theorem}
\begin{proof}
%Without loss of generality, let $b_{n}=n\geq 2$.
By Lemma $\ref{lemma:Qbreadthone}$,  $Q\subset \mathrm{span}\{L_{0},L_{1},\dots,L_{b_n}\}$ holds in the sense of coordinate transformation. By Proposition \ref{prop:1}, we know that $L_{m}$ can be obtained by specifying the parameters in $q_{n,m}$, thus the opposite inclusion holds and the theorem is proved.
\end{proof}

 \begin{corollary}
 $L_{k}, k=0,1,\dots,n$, in Theorem $\ref{theo:Linan}$ has the explicit expression:
 \begin{align*}
  L_{k}=\sum_{\tau(\pmb\gamma_{1},\dots,\pmb\gamma_{d})=k}\frac{1^{\gamma_{1,1}} a_{2,2}^{\gamma_{2,2}}\cdots a_{n,2}^{\gamma_{2,n}}\cdots a_{2,d}^{\gamma_{d,2}}\cdots a_{n,d}^{\gamma_{d,n}} }{\gamma_{1,1}!\gamma_{2,2}!\cdots \gamma_{2,n}!\cdots \gamma_{d,2}!\cdots \gamma_{d,n}!}x_{1}^{|\pmb\gamma_{1}|}\cdots x_{d}^{|\pmb\gamma_{d}|},
  \end{align*}
 where $\tau$ is defined by $(\ref{formu:tau})$.
%  \begin{align}
%  \tau(\gamma_{1},\dots,\gamma_{d})=1\cdot(\gamma_{1,1})+2\cdot(\gamma_{2,2}+\dots+\gamma_{d,2})+\dots+n\cdot (\gamma_{2,n}+\dots+\gamma_{d,n}).
%  \label{formu:Ltau}
%  \end{align}
  \label{coro:1}
 \end{corollary}

If we think of $L_{0},L_{1},\dots,L_{n}$ as a sequence of functions, then Theorem $\ref{theo:Linan}$ describes the recursion formula of this sequence of functions while the above corollary provides the general formula.

\section{The Discrete Approximation Problem for the Breadth-one $D$-invariant Subspace}\label{sec:4}
In this section, we will solve the discrete approximation problem for the special class of $D$-invariant subspaces $\mathrm{span}\{L_0,L_1,\dots,L_n \}$, which shows that an ideal interpolant, with its interpolation conditions of the form $\mathrm{span}\{\delta_{\mathbf{z}}\circ L_{0}(D),\dots,\delta_{\mathbf{z}}\circ L_{n}(D) \}$ is a limit of Lagrange interpolants.  The following lemma has been proved in \cite{van}, here we give a proof based on linear algebra.
\begin{lemma}\label{lemma:sum1}
  \cite{van} For any nonnegative integers $j$ and $m$,
  $$\sum_{i=0}^{m}(-1)^{m-i}\frac{1}{i!(m-i)!}i^{j}=\left\{
                                                \begin{array}{ll}
                                                  1, & \hbox{$j=m$;} \\
                                                  0, & \hbox{$0\leq j<m$.}
                                                \end{array}
                                              \right.
   $$
\end{lemma}
\begin{proof}
Consider the following linear equations:
$$\left(
  \begin{array}{cccc}
    0^{0} & 1^{0}  & \cdots & m^{0} \\
    0^{1} & 1^{1}  & \cdots & m^{1} \\
    \vdots &  \vdots & ~ & \vdots \\
    0^{m} & 1^{m}  & \cdots & m^{m} \\
  \end{array}
\right)\left(
         \begin{array}{c}
           y_{0} \\
           y_{1} \\
           \vdots \\
           y_{m} \\
         \end{array}
       \right)=\left(
                 \begin{array}{c}
                   0 \\
                   \vdots \\
                   0 \\
                   1 \\
                 \end{array}
               \right),
$$
using Cramer's rules, we get
$$y_{i}=\frac{(-1)^{m-i}}{i!(m-i)!}, ~~\forall i=0,1,\dots,m.$$
Thus the lemma is proved.
\end{proof}\\

Remove the first row and the first column of the coefficient matrix in the proof of the above lemma, we immediately obtain the following result.
\begin{lemma}\label{lemma:sum2}
  Let $j, m$ be nonnegative integers and $1\leq j\leq m$. Then
  $$\sum_{i=1}^{m}(-1)^{m-i}\frac{1}{i!(m-i)!}i^{j}=\left\{
                                                \begin{array}{ll}
                                                  1, & \hbox{$j=m$;} \\
                                                  0, & \hbox{$1\leq j<m$.}
                                                \end{array}
                                              \right.
   $$
\end{lemma}

%Next we give a set of discrete points for solving the discrete approximation problem for the breadth-one $D$-invariant subspace.
For any fixed $m>0$, we define
$$A^{(m)}_k:=(-1)^{m-k}\frac{1}{k!(m-k)!}, ~~k=0,\dots,m.$$
%The following theorem can be proved directly by Proposition $4.2$ in \cite{LiZhe} and Theorem $\ref{theo:3}$, we give a proof here for completeness.
\begin{theorem}\label{theorem:9}
 %Let $\mathcal{L}_{n}=\mathrm{span}\{L_{0},L_{1},\dots,L_{n}\}$ be the breadth-one $D$-invariant subspace.
Let $\mathbf{z}_{0}$ be a base point and $\mathcal{L}_{n}$ be as in Theorem $\ref{theorem:1}$. Define
  \begin{align}
  \mathbf{z}_{i}(h):=\mathbf{z}_{0}+(ih,\sum_{j=2}^{n}a_{j,2}(ih)^{j},\dots,\sum_{j=2}^{n}a_{j,d}(ih)^{j}),~~i=0,\dots,n.
  \label{formu:z}
  \end{align}
  Then for any function $f$ analytic at $\mathbf{z}_{0}$,
%and any fixed $m, m\in \{0,\dots,n\}$, there exist $A^{(m)}_r:=(-1)^{m-r}\frac{1}{r!(m-r)!}, r=0,\dots,m,$ such that
  $$(L_{m}(D)f)(\mathbf{z}_{0})=\lim_{h\rightarrow 0}\frac{1}{h^{m}}\left(\sum_{r=0}^{m}A^{(m)}_rf(\mathbf{z}_{r}(h))\right),~\forall m=0,\dots,n. $$
  In other words, the points defined in $(\ref{formu:z})$ are discrete points  satisfying
  \begin{align}
  \mathrm{span}\{\delta_{\mathbf{z}_{0}}\circ L_{0}(D),\dots,\delta_{\mathbf{z}_{0}}\circ L_{n}(D)\}=\lim_{h\rightarrow 0} \mathrm{span}\{\delta_{\mathbf{z}_{i}(h)},~i=0,\dots,n  \}.
  \label{formu:approximation}
  \end{align}
  \end{theorem}
\begin{proof}
 By  Taylor expansion and the definition of $\tau$ in (\ref{formu:tau}), we have
\begin{align*}
&f(\mathbf{z}_{i}(h))=\sum_{\Gamma=0}^{\infty}\left(ih\frac{\partial}{\partial x_1}+ \sum_{j=2}^{n}a_{j2}(ih)^{j}\frac{\partial}{\partial x_2}+\dots+\sum_{j=2}^{n}a_{jd}(ih)^{j}\frac{\partial}{\partial x_d}  \right)^{\Gamma}f(\mathbf{z}_{0})\\
&=\sum_{|\pmb\gamma_{1}|+\cdots+|\pmb\gamma_{d}|=0}^{\infty}\frac{1}{|\pmb\gamma_{1}|!\cdots|\pmb\gamma_{d}|!}\left(ih\right)^{|\pmb\gamma_{1}|}\left(\sum_{j=2}^{n}a_{j2}(ih)^{j}\right)^{|\pmb\gamma_{2}|}\cdots \left(\sum_{j=2}^{n}a_{jd}(ih)^{j}\right)^{|\pmb\gamma_{d}|}\frac{\partial^{|\pmb\gamma_{1}|+\cdots+|\pmb\gamma_{d}|}f(\mathbf{z}_{0})}{\partial x_{1}^{|\pmb\gamma_{1}|}\cdots \partial x_{d}^{|\pmb\gamma_{d}|}}\\
&=\sum_{s=0}^{m}\sum_{\tau(\pmb\gamma_{1},\dots,\pmb\gamma_{d})=s}(ih)^{s}\frac{1^{\gamma_{1,1}} a_{2,2}^{\gamma_{2,2}}\cdots a_{n,2}^{\gamma_{2,n}}\cdots a_{2,d}^{\gamma_{d,2}}\cdots a_{n,d}^{\gamma_{d,n}} }{\gamma_{1,1}!\gamma_{2,2}!\cdots \gamma_{2,n}!\cdots \gamma_{d,2}!\cdots \gamma_{d,n}!}\frac{\partial^{|\pmb\gamma_{1}|+\cdots+|\pmb\gamma_{d}|}f(\mathbf{z}_{0})}{\partial x_{1}^{|\pmb\gamma_{1}|}\cdots \partial x_{d}^{|\pmb\gamma_{d}|}}+O(h^{m+1}).
\end{align*}
Thus by Lemma $\ref{lemma:sum1}$, $\forall m=0,\dots,n$, we know that
\begin{align*}
&\sum_{r=0}^{m}A^{(m)}_{r}f(\mathbf{z}_{r}(h))\\
&=h^m \sum_{\tau(\pmb\gamma_{1},\dots,\pmb\gamma_{d})=m}\frac{1^{\gamma_{1,1}} a_{2,2}^{\gamma_{2,2}}\cdots a_{n,2}^{\gamma_{2,n}}\cdots a_{2,d}^{\gamma_{d,2}}\cdots a_{n,d}^{\gamma_{d,n}} }{\gamma_{1,1}!\gamma_{2,2}!\cdots \gamma_{2,n}!\cdots \gamma_{d,2}!\cdots \gamma_{d,n}!}\frac{\partial^{|\pmb\gamma_{1}|+\cdots+|\pmb\gamma_{d}|}f(\mathbf{z}_{0})}{\partial x_{1}^{|\pmb\gamma_{1}|}\cdots \partial x_{d}^{|\pmb\gamma_{d}|}}+O(h^{m+1})\\
&=h^m(L_m(D)f)(\mathbf{z}_{0})+O(h^{m+1}).
\end{align*}
Thus the theorem is proved.
\end{proof}\\

Finally, we will give another set of discrete points for $\mathrm{span}\{L_{0},L_{1},\dots,L_{n}\}$.

\begin{lemma}\label{lemma:10}
  For any fixed integers $r\geq 1, i\geq 2$,
\begin{align*}
 &\sum_{\alpha_1+2\alpha_{22}+\cdots+i\alpha_{2i}=r}i^{\alpha_1}\big[i(i-1)\big]^{\alpha_{22}}\big[i(i-1)(i-2)\big]^{\alpha_{23}}\cdots\big[i!\big]^{\alpha_{2i}}\\
&=\sum_{\alpha_1+2\alpha_{22}+\cdots+r\alpha_{2r}=r}i^{\alpha_1}\big[i(i-1)\big]^{\alpha_{22}}\cdots\big[i(i-1)\cdots(i-(r-1))\big]^{\alpha_{2r}}.
\end{align*}
\end{lemma}
\begin{proof}
If $i>r$, then $\alpha_{1}+2\alpha_{22}+\cdots+r\alpha_{2r}+\cdots+i\alpha_{2i}=r$ yields
$\alpha_{2,r+1}=\cdots=\alpha_{2i}=0, $
thus the equation holds.
 If $i<r$, then
\begin{align*}
&i^{\alpha_1}\big[i(i-1)\big]^{\alpha_{22}}\cdots\big[i(i-1)\cdots(i-(r-1))\big]^{\alpha_{2r}}\\
&=i^{\alpha_1}\big[i(i-1)\big]^{\alpha_{22}}\cdots\big[i!\big]^{\alpha_{2i}}
\big[i!\cdot 0\big]^{\alpha_{2,i+1}}\cdots\big[i!\cdot0\cdots(i-(r-1))\big]^{\alpha_{2r}}.
\end{align*}
Since  $0^{k}=0, \forall k\neq 0$ and $0^{0}=1$, it follows  that $$i^{\alpha_1}\big[i(i-1)\big]^{\alpha_{22}}\cdots\big[i(i-1)\cdots(i-(r-1))\big]^{\alpha_{2r}}\neq0$$ if and only if $$\alpha_{2,i+1}=\dots=\alpha_{2r}=0.$$   This completes the proof.
\end{proof}

\begin{theorem}\label{theorem:11}
 Let the base point be $\mathbf{z}_{0}$ and $d=2$, define
\begin{align*}
  \mathbf{z}_{0}(h)&:=\mathbf{z}_{0}, ~~\mathbf{z}_{1}(h):=\mathbf{z}_{0}+(h,0),\\
  \mathbf{z}_{i}(h)&:=\mathbf{z}_{0}+(ih,\sum_{j=2}^{i}i(i-1)\cdots(i-j+1)a_{j,2}h^{j}), ~ i=2,\dots,n.
\end{align*}
Then the set of points $\{\mathbf{z}_{0}(h),\mathbf{z}_{1}(h),\dots,\mathbf{z}_{n}(h)\}$ also satisfies $(\ref{formu:approximation})$.
\end{theorem}
\begin{proof}
For an arbitrary  $f$  analytic at $\mathbf{z}_{0}$, with Taylor expansion, we have
\begin{align*}
&f(\mathbf{z}_{1}(h))=\sum_{\alpha_{1}=0}^{\infty}\frac{1}{\alpha!}\frac{\partial ^{\alpha_{1}}f(\mathbf{z}_{0})}{\partial x_{1}^{\alpha_{1}}}h^{\alpha_{1}};\\
&f(\mathbf{z}_{i}(h))=\sum_{k=0}^{\infty}\sum_{\alpha_{1}+\alpha_{2}=k}\frac{1}{\alpha_{1}!\alpha_{2}!}(ih)^{\alpha_{1}}\left(\sum_{j=2}^{i}i(i-1)\cdots(i-j+1)
a_{j,2}h^{j}\right)^{\alpha_{2}}\frac{\partial ^{k}f(\mathbf{z}_{0})}{\partial x_{1}^{\alpha_{1}}\partial x_{2}^{\alpha_{2}}}\\
&=\sum_{k=0}^{\infty}\sum_{\alpha_{1}+\alpha_{2}=k}\sum_{\alpha_{22}+\cdots+\alpha_{2i}=\alpha_{2}}
\frac{a_{22}^{\alpha_{22}}\cdots a_{i2}^{\alpha_{2i}}}{\alpha_{1}!\alpha_{22}!\cdots\alpha_{2i}!}
i^{\alpha_{1}}(i(i-1))^{\alpha_{22}}\cdots(i!)^{\alpha_{2i}}h^{\alpha_{1}+2\alpha_{22}+\cdots+i\alpha_{2i}}\frac{\partial ^{k}f(\mathbf{z}_{0})}{\partial x_{1}^{\alpha_{1}}\partial x_{2}^{\alpha_{2}}}\\
&=\sum_{k=0}^{\infty}\sum_{\alpha_{1}+\alpha_{22}+\cdots+\alpha_{2i}=k}
\frac{a_{22}^{\alpha_{22}}\cdots a_{i2}^{\alpha_{2i}}}{\alpha_{1}!\alpha_{22}!\cdots\alpha_{2i}!}
i^{\alpha_{1}}(i(i-1))^{\alpha_{22}}\cdots(i!)^{\alpha_{2i}}
h^{\alpha_{1}+2\alpha_{22}+\cdots+i\alpha_{2i}}\frac{\partial ^{k}f(\mathbf{z}_{0})}{\partial x_{1}^{\alpha_{1}}\partial x_{2}^{\alpha_{22}+\cdots+\alpha_{2i}}}
\end{align*}
for $i=2,\dots,n$. For convenience, we will write
\begin{align*}
i^{\alpha_{1}}(i(i-1))^{\alpha_{22}}\cdots(i(i-1)\cdots(i-(r-1)))^{\alpha_{2r}}\triangleq i^{\alpha_{1}+2\alpha_{22}+\dots+r\alpha_{2r}}+\omega(i),
\end{align*}
where $\omega(i)$ is a polynomial in $i$ with degree less than  $\alpha_{1}+2\alpha_{22}+\dots+r\alpha_{2r}$.

%For any fixed $m$, $1\leq m \leq n$, define $A_{r}^{(m)}:=\frac{(-1)^{m-r}}{m!(m-r)!}, r=0,\dots,m.$
For any fixed $m\in \{0,\dots,n \}$, consider the following combination
\begin{align}
\sum_{i=0}^{m}A_{i}^{(m)}f(\mathbf{z}_{i}(h))\triangleq\sum_{i=0}^{m}W_{i}h^{i}+O(h^{m+1}),
\end{align}
where
\begin{align*}
W_{0}&=\sum_{i=0}^{m}A_{i}^{(m)};\\
W_{1}&=\sum_{i=1}^{m}A_{i}^{(m)}\sum_{\alpha_{1}=1}\frac{i^{\alpha_{1}}}{\alpha_{1}!}\frac{\partial^{\alpha_{1}}f(\mathbf{z}_{0})}
{\partial x_{1}^{\alpha_{1}} }=\sum_{i=1}^{m}A_{i}^{(m)}i\frac{\partial f(\mathbf{z}_{0})}
{\partial x_{1} };\\
%W_{r}&=A_{1}^{(m)}\sum_{\alpha_{1}=r}\frac{1^{\alpha_{1}}}{\alpha_{1}!}\frac{\partial^{\alpha_{1}}}{\partial x_{1}^{\alpha_{1}}}
W_{r}&=\sum_{i=1}^{m}A_{i}^{(m)}\sum_{\alpha_{1}+2\alpha_{22}+\cdots+i\alpha_{2i}=r}\frac{1}{\alpha_{1}!}\frac{a_{22}^{\alpha_{22}}\cdots a_{i2}^{\alpha_{2i}}}{\alpha_{22}!\cdots \alpha_{2i}!}i^{\alpha_{1}}(i(i-1))^{\alpha_{22}}\cdots(i!)^{\alpha_{2i}} \frac{\partial^{\alpha_{1}+\alpha_{22}+\cdots+\alpha_{2i}}f(\mathbf{z}_{0})}{\partial x_{1}^{\alpha_{1}}\partial x_{2}^{\alpha_{22}+\cdots+\alpha_{2i}}}\\
&=\sum_{i=1}^{m}A_{i}^{(m)}\sum_{\alpha_{1}+2\alpha_{22}+\cdots+r\alpha_{2r}=r}\bigg(\frac{1}{\alpha_{1}!}\frac{a_{22}^{\alpha_{22}}\cdots a_{r2}^{\alpha_{2r}}}{\alpha_{22}!\cdots \alpha_{2r}!}\\
&~~~~~~~~~~~~~~~~~~~~~~~\cdot i^{\alpha_{1}}(i(i-1))^{\alpha_{22}}\cdots(i(i-1)\cdots(i-(r-1)))^{\alpha_{2r}} \frac{\partial^{\alpha_{1}+\alpha_{22}+\cdots+\alpha_{2r}}f(\mathbf{z}_{0})}{\partial x_{1}^{\alpha_{1}}\partial x_{2}^{\alpha_{22}+\cdots+\alpha_{2r}}}\bigg)\\
&=\sum_{i=1}^{m}A_{i}^{(m)}\sum_{\alpha_{1}+2\alpha_{22}+\cdots+r\alpha_{2r}=r}\frac{1}{\alpha_{1}!}\frac{a_{22}^{\alpha_{22}}\cdots a_{r2}^{\alpha_{2r}}}{\alpha_{22}!\cdots \alpha_{2r}!}(i^{r}+\omega(i)) \frac{\partial^{\alpha_{1}+\alpha_{22}+\cdots+\alpha_{2r}}f(\mathbf{z}_{0})}{\partial x_{1}^{\alpha_{1}}\partial x_{2}^{\alpha_{22}+\cdots+\alpha_{2r}}}\\
&=\sum_{\alpha_{1}+2\alpha_{22}+\cdots+r\alpha_{2r}=r}\frac{1}{\alpha_{1}!}\frac{a_{22}^{\alpha_{22}}\cdots a_{r2}^{\alpha_{2r}}}{\alpha_{22}!\cdots \alpha_{2r}!}\sum_{i=1}^{m}A_{i}^{(m)}(i^{r}+\omega(i)) \frac{\partial^{\alpha_{1}+\alpha_{22}+\cdots+\alpha_{2r}}f(\mathbf{z}_{0})}{\partial x_{1}^{\alpha_{1}}\partial x_{2}^{\alpha_{22}+\cdots+\alpha_{2r}}}
\end{align*}
for $2\leq r\leq m$, here we define $\alpha_{21}=0$. The second equation in $W_{r}$ holds according to Lemma \ref{lemma:10}.

By Corollary $\ref{coro:1}$ and Lemma $\ref{lemma:sum1}$, $\ref{lemma:sum2}$, we know that
\begin{align*}
&\sum_{i=0}^{m}A_{i}^{(m)}f(\mathbf{z}_{i}(h))\\
&=h^{m}\sum_{\alpha_{1}+2\alpha_{22}+\cdots+m\alpha_{2m}=m}\frac{1}{\alpha_{1}!}\frac{a_{22}^{\alpha_{22}}\cdots a_{m2}^{\alpha_{2m}}}{\alpha_{22}!\cdots \alpha_{2m}!}\frac{\partial^{\alpha_{1}+\alpha_{22}+\cdots+\alpha_{2m}}f(\mathbf{z}_{0})}{\partial x_{1}^{\alpha_{1}}\partial x_{2}^{\alpha_{22}+\cdots+\alpha_{2m}}}+O(h^{m+1}) \\
&= h^{m}\sum_{\alpha_{1}+2\alpha_{22}+\cdots+n\alpha_{2n}=m}\frac{1}{\alpha_{1}!}\frac{a_{22}^{\alpha_{22}}\cdots a_{n2}^{\alpha_{2n}}}{\alpha_{22}!\cdots \alpha_{2n}!}\frac{\partial^{\alpha_{1}+\alpha_{22}+\cdots+\alpha_{2n}}f(\mathbf{z}_{0})}{\partial x_{1}^{\alpha_{1}}\partial x_{2}^{\alpha_{22}+\cdots+\alpha_{2n}}}+O(h^{m+1})\\
&=h^{m}(L_{m}(D)f)(\mathbf{z}_{0})+O(h^{m+1}),
\end{align*}
thus
$$\lim_{h\rightarrow 0}\frac{1}{h^{m}}\left(\sum_{i=0}^{m}A_{i}^{(m)}f(\mathbf{z}_{i}(h))\right)=(L_{m}(D)f)(\mathbf{z}_{0}), ~\forall m=2,\dots,n,$$
that is,
$$\delta_{\mathbf{z}_{0}}\circ L_{m}(D)=\lim_{h\rightarrow 0}\frac{1}{h^{m}}\left(\sum_{i=0}^{m}A_{i}^{(m)}\delta_{\mathbf{z}_{i}(h)}\right).  $$
It is easy to verify that the above equation also holds for $m=1$. Thus the theorem is proved.
\end{proof}\\

For $d\geq 3$, we give the following result without a proof.
\begin{theorem}
 Suppose that the base point is $\mathbf{z}_{0}$, define
  $\mathbf{z}_{0}(h):=\mathbf{z}_{0}, \mathbf{z}_{1}(h):=\mathbf{z}_{0}+(h,0,\dots,0)$ and for $i=2,\dots,n$,
$$\mathbf{z}_{i}(h):=\mathbf{z}_{0}+(ih,\sum_{j=2}^{i}\frac{i!}{(i-j)!}a_{j,2}h^{j},\sum_{j=2}^{i}\frac{i!}{(i-j)!}a_{j,3}h^{j},\dots,\sum_{j=2}^{i}\frac{i!}{(i-j)!}a_{j,d}h^{j}). $$
Then $\{\mathbf{z}_{0}(h),\mathbf{z}_{1}(h),\dots,\mathbf{z}_{n}(h)\}$ is  a set  of discrete points for the breadth-one subspace $\mathcal{L}_n=\mathrm{span}\{L_0,L_1,\dots,L_n\}$.
\end{theorem}

\noindent\textbf{Example~1}. Let $d=2, n=4, a_{2,2}=2, a_{3,2}=3, a_{4,2}=4$, thus,
 $$\mathcal{L}_4=\mathrm{span}\{1, x_1, \frac{1}{2}x_1^2+2x_2, \frac{1}{3!}x_1^3+2x_1x_2+3x_2,\frac{1}{4!}x_1^4+x_1^2x_2+3x_1x_2+2x_2^2+4x_2 \}.$$
Let $\mathbf{z}_{0}:=(0,0)$, then by Theorem $\ref{theorem:9}$ and Theorem $\ref{theorem:11}$,
\begin{align*}
\{&(0,0),~(h,0),~(2h,2(2h)^2+3(2h)^3+4(2h)^4),\\
&(3h,2(3h)^2+3(3h)^3+4(3h)^4),~(4h,2(4h)^2+3(4h)^3+4(4h)^4)\}
\end{align*}
and
\begin{align*}
 \{&(0,0),~(h,0),~(2h,4h^2),\\
&(3h,12h^2+18h^3),~(4h,24h^2+72h^3+96h^4)\}
\end{align*}
are two sets of discrete points satisfying $(\ref{formu:approximation})$.

\end{document}